\documentclass[11pt]{article}

\usepackage{fullpage}
\usepackage[T1]{fontenc} 
\usepackage{amsmath,amssymb}
\usepackage{amsopn,amsthm}

\usepackage{graphicx} 
\usepackage{color}
\usepackage{psfrag}

\usepackage{mathrsfs} 
\usepackage{bm}

\usepackage{algpseudocode} 
\usepackage{algorithm}
\floatstyle{plain} \newfloat{myalgo}{tbhp}{mya}

{\begin{myalgo}[#1]
    \centering
    \begin{minipage}{0.9\textwidth}
      \begin{algorithm}[H]}%
      {\end{algorithm}
    \end{minipage}
  \end{myalgo}}

\usepackage[colorlinks=true]{hyperref}
\hypersetup{urlcolor=blue, citecolor=blue, linkcolor=blue}

\newtheorem{theorem}{Theorem}[section]
\newtheorem{proposition}{Proposition}[section]

\theoremstyle{definition}

\newtheorem{example}[theorem]{Example}

\theoremstyle{remark} 
\newtheorem{remark}[theorem]{Remark}

\numberwithin{equation}{section}
\numberwithin{figure}{section}
\numberwithin{algorithm}{section}

\newcommand{\field}[1]{{\mathbb{#1}}}
\newcommand{\C}{\field{C}} 

\newcommand{\R}{\field{R}}

\newcommand{\Dcal}{\mathcal{D}}

\newcommand{\Jcal}{\mathcal{J}}

\newcommand{\Mcal}{\mathcal{M}}

\newcommand{\Rcal}{\mathcal{R}}

\newcommand{\loc}{{\mathrm{loc}}}

\newcommand{\tm}{\subseteq} 
\newcommand{\di}{\partial}
\newcommand{\trans}{{\top}}

\newcommand{\dy}{\, \dif y}

\newcommand{\deta}{\, \dif\eta}

\newcommand{\qhat}{\widehat q}

\newcommand{\xhat}{\hat x}

\newcommand{\qtilde}{{\widetilde q}}

\newcommand{\utilde}{{\widetilde u}}
\newcommand{\vtilde}{{\widetilde v}}

\newcommand{\Ltilde}{{\widetilde L}}
\newcommand{\Mtilde}{{\widetilde M}}

\newcommand{\Ptilde}{{\widetilde P}}

\newcommand{\sigmatilde}{{\widetilde\sigma}}

\newcommand{\rmi}{\mathrm{i}}

\newcommand{\mmin}{\mathrm{min}}

\newcommand{\kmin}{k_{\min}}
\newcommand{\kmax}{k_{\max}}
\newcommand{\lambdamin}{\lambda_{\min}}
\newcommand{\lambdamax}{\lambda_{\max}}

\newcommand{\eps}{\varepsilon}

\newcommand{\uinfty}{u^\infty}

\newcommand{\ph}{\,\cdot\,}

\DeclareMathAlphabet{\mathbi}{\encodingdefault}{\rmdefault}{\bfdefault}{\itdefault}
\DeclareRobustCommand{\vec}[1]{\ifmmode\mathbi{#1}\else\textbf{\textit{#1}}\fi}

\DeclareMathOperator{\dif}{d\!}

\DeclareMathOperator{\supp}{supp}

\DeclareMathOperator{\rank}{rank}

\begin{document}

\title{A multifrequency MUSIC algorithm for locating small
  inhomogeneities in inverse scattering} 
\author{Roland Griesmaier and 
  Christian Schmiedecke\footnote{Institut f\"ur Mathematik,
    Universit\"at W\"urzburg, 97074 W\"urzburg, Germany
    ({\tt roland.griesmaier@uni-wuerzburg.de}, 
    {\tt christian.schmiedecke@mathematik.uni-wuerzburg.de})} 
}
\date{\today}

\maketitle

\begin{abstract}
  We consider an inverse scattering problem for time-harmonic acoustic or electromagnetic waves. The goal is to localize several small penetrable objects embedded inside an otherwise homogeneous background medium from observations of far field patterns of scattered fields corresponding to plane wave incident fields with one fixed incident direction but several different frequencies. Taking advantage of the smallness of the scatterers with respect to wave length we utilize an asymptotic representation formula for the far field pattern to design and analyze a MUSIC-type reconstruction method for this setup. We establish lower bounds on the number of frequencies and receiver directions required to recover the number and the positions of the scatterers for a given configuration by the reconstruction algorithm. Furthermore we apply the method to the practically interesting case of multifrequency backscattering data. Numerical examples are presented to document the potentials and limitations of this approach.
\end{abstract}

{\small\noindent
  Mathematics subject classifications (MSC2010): 35R30, (65N21)
  \\\noindent
  Short title: Multifrequency inverse scattering
}

\section{Introduction}
\label{sec:Introduction}
The MUltiple SIgnal Classification (MUSIC) algorithm from signal processing has in the past years been successfully utilized as a non-iterative reconstruction method for detecting small point-like scattering obstacles or inhomogeneities from so-called multi-static scattering data (cf.~\cite{AmmIakLes04,Che01,Dev00,Dev12,Kir02,Sch86}), and it is well known that MUSIC-type reconstruction methods have strong connections to other qualitative reconstruction schemes such as, e.g., the linear sampling method~\cite{ColKir96} or the factorization method~\cite{Kir98} (see \cite{AmmGriHan07,Che01}). 
The basic assumption of this class of algorithms is that observations of scattered fields generated by sufficiently many different incident fields at a fixed frequency (i.e., multi-static scattering data) are available.

In this work we consider a different data set.
We assume that far field patterns of scattered fields corresponding to plane wave incident fields with just one fixed incident direction but several different frequencies are given.
We design and analyze a \emph{multifrequency MUSIC} algorithm for determining the number and the positions of finite collections of small scattering objects from these data.
We will also show how to adapt this reconstruction scheme to work with multifrequency backscattering data sets.

Our discussion of the multifrequency MUSIC algorithm is based on an asymptotic representation formula for the far field pattern of the scattered field as the size of the scatterers tends to zero with respect to the wave length of the incident field.
Such expansions have been widely studied in the literature (cf.~\cite{AmmIakMos04,VogVol00}).
However, since we apply the result to justify a multifrequency reconstruction scheme, we put particular emphasis on frequency dependence and give a short concise proof showing that the asymptotic formula holds uniformly across whole frequency bands.
Moreover, we include the possibility that the scatterers are absorbing, which is modeled by a frequency dependent imaginary part in the index of refraction.

Making use of the fact that the leading order term in the asymptotic expansion of the far field pattern resembles a so-called \emph{multivariate extended exponential sum}, the inverse scattering problem essentially reduces to a parameter estimation problem.
Related problems have recently been considered in \cite{KunPetRomOhe16,PloWis13,PotTas13} where different multivariate generalizations of Prony-type methods have been proposed,\footnote{For a recent survey on univariate Prony-type methods we refer to \cite{PloTas14}.} in \cite{Lia15,LiaFan16} where MUSIC-type methods for single-snapshot spectral estimation have been considered, and in \cite{CanFer14} where a rather different approach relying on total variation minimization has been discussed.
The multifrequency MUSIC algorithm considered in this work generalizes the one-dimensional scheme considered in \cite{LiaFan16} and is also related to multi-variate Prony-type methods. 
When compared to the methods in \cite{KunPetRomOhe16,PloWis13,PotTas13} our scheme has two advantages: (i) it does not require to compute common zeros of sets of multivariate polynomials, and (ii) we do not have to assume that the coefficients in the multivariate exponential sum have equal sign to avoid cancellation effects.
Indeed, we give sufficient conditions on the number of frequencies and observation directions that are needed to guarantee that the reconstruction algorithm recovers all scattering objects---at least theoretically---which generalizes related results on parameter estimation for bivariate exponential sums in \cite{DieIsk15,DieIsk16}.

Recently a multifrequency MUSIC algorithm has been considered for the (quasi-stationary) inverse boundary value problem of electrical impedance tomography in \cite{ABG13,GriHan15}, but the approach used in these works differs considerably from the method discussed here.
Furthermore, a multifrequency linear sampling method has been proposed in \cite{GuzCakBel10}, but in contrast to our work this method requires a much larger data set, namely observations of scattered fields generated by plane wave incident fields with sufficiently many different incident directions at several different frequencies.
On the other hand, a Prony-type reconstruction method for a much smaller data set that requires a single far field pattern corresponding to a plane wave incident field at a single frequency only has been considered in~\cite{Han12}.
Since this method uses the evanescent modes of the far field pattern only, it suffers from a certain lack of stability.
Finally, backpropagation algorithms for multifrequency scattering data have been analyzed in \cite{Gri11,Luk04}.

The paper is organized as follows.
In the next section we specify the problem setting and establish the asymptotic representation formula for the far field pattern due to an ensemble of finitely many small scattering objects.
The multifrequency MUSIC algorithm for determining the number and the positions of the scatterers is discussed in Section~\ref{sec:MUSIC}.
Then, in Section~\ref{sec:Numerics} we provide numerical results and briefly comment on how to adapt the multifrequency MUSIC scheme to work with backscattering data.
We conclude with some final remarks.

\section{Problem setting}
\label{sec:Setting}
Let
\begin{equation}
  \label{eq:DefD}
  D \,:=\, \bigcup_{m=1}^M D_m \,\tm\, \R^d
\end{equation}
be a collection of finitely many well-separated bounded domains in $\R^d$, $d = 2,3$.  
Each subdomain $D_m$, $m=1,\ldots,M$, represents a scattering object specified by its \emph{refractive index} $n_m$, all of which are embedded in a homogeneous background medium such that the refractive index $n$ of the whole configuration is given by
\begin{equation*}
  n \,=\,
  \begin{cases}
    n_m \,, &\text{in } D_m \,,\; m=1,\ldots,M \,,\\
    1 \,, &\text{in } \R^d\setminus D \,.
  \end{cases}
\end{equation*}
Allowing for absorbing scatterers, we assume that
\begin{equation}
  \label{eq:Defnl}
  n_m \,=\, 1 + q_{1,m} + \frac{\rmi}{k} q_{2,m} 
\end{equation}
for some $q_{1,m},q_{2,m} \in L^\infty(D_m)$ satisfying $q_{1,m}>-1$, $q_{2,m}\geq 0$, and $q_{1,m}+\frac{\rmi}{k}q_{2,m} \not= 0$ a.e.\ in $D_m$ (cf.~\cite[p.~268]{ColKre13}).
Here and throughout $k>0$ denotes the \emph{wave number}, and for later reference we also introduce the radius
\begin{equation}
  \label{eq:DefR}
  R := \max \{|x| \; : \; x \in D \}
\end{equation}
of the smallest ball around the origin enclosing $D$.

Next, let $u^i(\ph,\theta;k)$ be a time-harmonic \emph{plane wave incident field} with wave number $k>0$, i.e.,
\begin{equation*}
  u^i(x,\theta;k) \,=\, e^{\rmi k x\cdot\theta} \,, \qquad x\in\R^d \,,
\end{equation*}
where $\theta \in S^{d-1}$ indicates the \emph{direction of incidence}.
Then the inhomogeneous medium gives rise to a \emph{total field} $u(\ph,\theta;k)$ that satisfies
\begin{subequations}
  \label{eq:DirectProblem}
  \begin{gather}
    \Delta u + k^2 n u \,=\, 0 \qquad \text{in } \R^d \,,\\
    u = u^i+u^s \,,\\
    \lim_{r\to\infty} r^{\frac{1-d}{2}} \Bigl( \frac{\di u^s}{\di r} - \rmi k u^s \Bigr) \,=\, 0 \,,\qquad r = |x| \,,
  \end{gather}
\end{subequations}
where $u^s(\ph,\theta;k)$ denotes the \emph{scattered field}.
It is well known that this direct scattering problem has a unique solution $u(\ph,\theta;k) \in H^1_\loc(\R^d)$ (cf.~\cite{ColKre13,Kir11}).
Furthermore, the scattered field admits an asymptotic expansion at infinity of the form
\begin{equation*}
  u^s(x,\theta;k) \,=\, C_k \frac{e^{\rmi k |x|}}{|x|^{\frac{d-1}{2}}} \uinfty(\xhat,\theta;k) + O(|x|^{\frac{d+1}{2}}) \,, \qquad \xhat = \frac{x}{|x|} \in S^{d-1} \,, \; r = |x| \to \infty \,,
\end{equation*}
where $C_k = {e^{\rmi\pi/4}}/{\sqrt{8\pi k}}$ if $d=2$ and $C_k = k^2/(4\pi)$ if $d=3$.
Here, $\uinfty(\ph,\theta;k)$ denotes the \emph{far field pattern} of $u^s(\ph,\theta;k)$ which is given by
\begin{equation}
  \label{eq:FarFieldPattern}
  \uinfty(\xhat,\theta;k) \,=\, \int_{D} (n-1)(y) e^{-\rmi k \xhat\cdot y} u(y,\theta;k) \dy \,, \qquad \xhat \in S^{d-1} \,.
\end{equation}

Our aim is to determine information on the support of the scatterers $D_1,\ldots,D_M$ from observations of
\begin{equation}
  \label{eq:MeasurementData}
  \uinfty(\xhat_j,\theta;k_n) \,, \qquad j=1,\ldots,J \,,\; n=1,\ldots,2N \,,
\end{equation}
for $J>0$ mutually distinct \emph{receiver directions} $\xhat_j \in S^{d-1}$ and $2N>0$ different wave numbers $0 < k_1 < \cdots < k_{2N}$.
Note that in this work the incident direction $\theta \in S^{d-1}$ of $u^i$ is fixed, but the wave number $k$ varies.

We analyze the inverse problem in a restricted setting, assuming that the supports of the scatterers $D_m$ are of small diameter.
To this end we write
\begin{equation}
  \label{eq:RescaledDomain}
  D_m \,=\, z_m + \eps B_m \,, \qquad m=1,\ldots,M \,,
\end{equation}
where $z_m \in \R^d$, $m = 1,\ldots,M$, are the mutually distinct centers of mass of $D_m$, and the scaling parameter $\eps > 0$ is sufficiently close to zero.
We refer to $z_m$ and $D_m$ as the \emph{position} and the \emph{shape} of the scatterer $D_m$, respectively.
For consistency we assume in the following that
\begin{equation}
  \label{eq:RefractiveIndex}
  n_m(x) \,=\, 1+\qhat_{1,m}\Bigl(\frac{x-z_m}{\eps}\Bigr) + \frac{\rmi}{k}\qhat_{2,m}\Bigl(\frac{x-z_m}{\eps}\Bigr) \qquad \text{for some }\; \qhat_{1,m}, \qhat_{2,m} \in L^\infty(B_m) \,.
\end{equation}
Then the far field data from \eqref{eq:FarFieldPattern} have the following asymptotic representation.

\begin{theorem}
  \label{thm:Asy}
  Assume that $0<\kmin\leq k<\infty$ and let $n_m(x)$ as in \eqref{eq:RefractiveIndex} with $\qhat_{1,m}>-1$ and $\qhat_{2,m}\geq0$ a.e.\ in $D_m$, $m=1,\ldots,M$.\footnote{We require the lower bound $k \geq \kmin > 0$ to avoid degeneracy in \eqref{eq:Defnl} only.  If the scatterers are non-absorbing, then $k>0$ is sufficient.}
  Then, for $\xhat \in S^{d-1}$ the far field pattern satisfies
  \begin{equation}
    \label{eq:Asy}
    k^{d-2} \uinfty(\xhat,\theta;k) 
    \,=\, (k \eps)^d \sum_{m=1}^M \Bigl( \qtilde_{1,m}(\eta) + \frac{\rmi}{k} \qtilde_{2,m}(\eta) \Bigr) e^{\rmi k (\theta-\xhat)\cdot z_m} + O((k\eps)^{d+1})
  \end{equation}
  as $k\eps \to 0$, where
  \begin{equation*}
    \qtilde_{j,m} \,=\, \int_{B_m} \qhat_{j,m}(\eta) \deta \,=\, \frac1{\eps^d} \int_{D_m} q_{j,m}(y) \dy \,, \qquad j=1,2 \,.
  \end{equation*}
  The last term on the right hand side of \eqref{eq:Asy} is bounded by $C(k\eps)^{d+1}$, uniformly in all directions $\xhat \in S^{d-1}$.
  Here, the constant $C>0$ depends on $\kmin$ but is independent of $k$ and $\eps$.
\end{theorem}

\begin{proof}
  Rewriting \eqref{eq:DirectProblem} as
  \begin{equation*}
    \Delta u^s + k^2 u^s \,=\, - k^2 (n-1) u \qquad \text{in } \R^d \,,
  \end{equation*}
  it follows immediately that the total field satisfies the Lippmann-Schwinger integral equation
  \begin{equation*}
    u(x,\theta;k) \,=\, u^i(x,\theta;k) + k^2 \int_D (n-1)(y) \Phi_k(x-y) u(y,\theta;k) dy \,,
  \end{equation*}
  where
  \begin{equation*}
    \Phi_k (x) \,=\,
    \begin{cases}
      \displaystyle \frac\rmi4 H^{(1)}_0(k|x|) \,, & d = 2 \,,\\[1em]
      \displaystyle \frac{e^{\rmi k |x|}}{4\pi|x|} \,, & d = 3 \,,\\
    \end{cases}
    \qquad\qquad x \in \R^d\setminus \{0\} \,,
  \end{equation*}
  denotes the fundamental solution to the Helmholtz equation.

  Given any $f \in C(D)$, we find that $(n-1)f \in L^\infty(D)$, and accordingly the volume potential $\int_D (n-1)(y) \Phi_k(x-y) f(y) \dy$ is continuously differentiable throughout $\R^d$ (cf.~\cite[p.~54]{GilTru01}).
  Hence the operator $T: C(D) \to C(D)$\,,
  \begin{equation*}
    (T f)(x) \,:=\, k^2 \int_D (n-1)(y) \Phi_k(k |x-y|) f(y) dy \,, \qquad x \in D \,,
  \end{equation*}
  is well-defined and recalling \eqref{eq:DefD}
  \begin{equation*}
    |(T f)(x)| \,\leq\, k^2 \|n-1\|_{L^\infty(D)} \|f\|_{C(D)}
    \begin{cases}
      \displaystyle \frac14 \sum_{m=1}^M \int_{D_m} |H^{(1)}_0(k|x-y|)| \dy \,, &d=2 \,,\\[1em]
      \displaystyle \frac1{4\pi} \sum_{m=1}^M \int_{D_m} \frac1{|x-y|} \dy \,, &d=3 \,.
    \end{cases}
  \end{equation*}
  For $d=2$, we can combine the fact that $\sqrt{t}|H^{(1)}_0(t)|$ is monotonically increasing for $t>0$ (cf.~\cite[p.~446]{Wat44}) with the limiting behavior of this function
  \begin{equation*}
    \lim_{t\to0} |\sqrt{t} H^{(1)}_0(t)| \,=\, 0 \qquad \text{ and } \qquad
    \lim_{t\to\infty} |\sqrt{t} H^{(1)}_0(t)| \,=\, \sqrt{\frac2\pi} \,,
  \end{equation*}
  (cf.~\cite[p.~74]{ColKre13}) to obtain the upper bound
  \begin{equation*}
    \sup_{t\geq 0}{\sqrt{t}|H^{(1)}_0(t)|} = \sqrt{\frac2\pi} \,.
  \end{equation*}
  Thus,
  \begin{equation*}
    \int_{D_m} |H^{(1)}_0(k|x-y|)| \dy \,=\,
    \begin{cases}
      O(\eps^2) &\text{if } x \in D_{l}\,,\; l\not= m \,,\\
      O(k^{-1/2}\eps^{3/2}) &\text{if } x \in D_m \,.
    \end{cases}
  \end{equation*}
  Similarly, we obtain for $d=3$ 
  \begin{equation*}
    \int_{D_m} \frac1{|x-y|} \dy \,=\,
    \begin{cases}
      O(\eps^d) &\text{if } x \in D_{l}\,,\; l\not= m \,,\\
      O(\eps^{d-1}) &\text{if } x \in D_m \,.
    \end{cases}
  \end{equation*}
  Therefore,
  \begin{equation*}
    \|T\| \,\leq\, C \|n-1\|_{L^\infty(D)}
    \begin{cases}
      k^{3/2}\eps^{3/2} &\text{if } d=2\,,\\
      k^2\eps^2 &\text{if } d=3\,,
    \end{cases}
  \end{equation*}
  where $\|\cdot\|$ denotes the operator norm on the space of bounded linear operators on $C(D)$.
  We note that the assumption $k\geq \kmin>0$ guarantees that the term $\|n-1\|_{L^\infty(D)}$ remains bounded.

  Suppose next that $k\eps>0$ is small enough such that $\|T\|\leq\frac12$. 
  Then $I-T:C(D)\to C(D)$ is invertible, and $\|(I-T)^{-1}\| \leq 2$.
  In particular,
  \begin{equation*}
    \|u^s\|_{C(D)} \,=\, \|Tu\|_{C(D)} \,=\, \|T(I-T)^{-1}u^i\|_{C(D)} 
    \,\leq\, 2 \|T\| \|u^i\|_{C(D)} \,=\,
    \begin{cases}
      O(k^{3/2}\eps^{3/2}) &\text{if } d=2\,,\\
      O(k^2\eps^2) &\text{if } d=3\,.
    \end{cases}
  \end{equation*}

  Recalling \eqref{eq:FarFieldPattern} we find that
  \begin{equation*}
    \uinfty(\xhat,\theta;k) 
    \,=\, k^2 \int_D (n-1)(y) e^{-\rmi k \xhat\cdot y} u^i(y,\theta;k) \dy
    + k^2 \int_D (n-1)(y) e^{-\rmi k \xhat\cdot y} u^s(y,\theta;k) \dy \,,
  \end{equation*}
  and using Taylor expansion we obtain for the first integral that
  \begin{equation*}
    \begin{split}
      k^2 \int_D (n-1)(y) e^{-\rmi k \xhat\cdot y}& u^i(y,\theta;k) \dy\\
      &\,=\, k^2 \sum_{m=1}^M \int_{B_m} \Bigl( \qhat_{1,m}(\eta) + \frac{\rmi}{k} \qhat_{2,m}(\eta) \Bigr) \bigl( e^{\rmi k (\theta-\xhat)\cdot z} + O(k\eps) \bigr) \eps^d \deta\\
      &\,=\, k^2 \eps^d \sum_{m=1}^M e^{\rmi k (\theta-\xhat)\cdot z} \int_{B_m} \Bigl( \qhat_{1,m}(\eta) + \frac{\rmi}{k} \qhat_{2,m}(\eta) \Bigr) \deta + O(k^3\eps^{d+1}) \,.
    \end{split}
  \end{equation*}
  The second integral satisfies
  \begin{equation*}
    \begin{split}
      \Bigl| k^2 \int_D (n-1)(y) e^{-\rmi k \xhat\cdot y} u^s(y,\theta;k) \dy \Bigr|
      &\,\leq\, k^2 \|n-1\|_{L^\infty(D)} \|u^s\|_{C(D)} \int_D 1 \dy\\
      &\,=\, \begin{cases}
        O(k^{7/2}\eps^{7/2}) &\text{if } d=2\,,\\
        O(k^4\eps^5) &\text{if } d=3\,.
      \end{cases}
    \end{split}
  \end{equation*}
  This ends the proof.
\end{proof}

\section{The multifrequency MUSIC scheme}
\label{sec:MUSIC}
Together with our basic assumption that the scattering objects are of sufficiently small diameter with respect to wave length, parameterized by the value of $\eps>0$ in \eqref{eq:RescaledDomain}, Theorem~\ref{thm:Asy} implies that the rescaled far field patterns $\frac1{\eps^d k_n}\uinfty(\ph,\theta;k_n)$, $n=1,\ldots,2N$, approximate a \emph{multivariate extended exponential sum}
\begin{equation}
  \label{eq:emes}
  \frac1{\eps^d k_n} \uinfty(\xhat,\theta;k_n) 
  \,\approx\, h(\xhat,\theta; k_n) 
  \,:=\, \sum_{m=1}^M (k_n \qtilde_{1,m} + \rmi \qtilde_{2,m}) e^{\rmi k_n(\theta-\xhat)\cdot z_m} 
\end{equation}
if $k_n\eps$ is sufficiently small.

In the following we assume that $k_n\qtilde_{1,m}+\rmi \qtilde_{2,m} \not= 0$ for all $m = 1,\ldots, M$, since otherwise the corresponding small scatterer would not contribute significantly to the far field data, and we define
\begin{equation}
  \label{eq:DefM'}
  M' \,:=\, M + |\{ m \; : \; \qtilde_{1,m} \not= 0 \}|\leq 2M\, .
\end{equation}
In practice usually neither the number of scatterers $M$ nor $M'$ is known a priori, but we assume that at least an upper bound $L \geq M'$ is available.
Furthermore, the data $\uinfty(\xhat_j,\theta;k_n)$ in \eqref{eq:MeasurementData} shall be given for $2N$ equidistant frequencies
\begin{equation}
  \label{eq:WaveNumbers1}
  k_n \,=\, n \kmin \,, \qquad n = 1,\ldots,2N \,,
\end{equation}
where
\begin{equation}
  0< \kmin\leq \frac{\pi}{2R} \quad \text{with $R>0$ from \eqref{eq:DefR}} \qquad \text{and} \qquad  N > L \geq M' \,. 
  \label{eq:WaveNumbers2}
\end{equation}
The upper bound on $\kmin$ in \eqref{eq:WaveNumbers2} guarantees that $|\kmin(\theta-\xhat_j)\cdot z_m| < \pi$ for all $m=1,\ldots,M$, i.e., $\kmin(\theta-\xhat_j)\cdot z_m$ is uniquely determined by $e^{\rmi \kmin(\theta-\xhat_j)\cdot z_m}$.
We note that the wave length $\lambda_\mmin$ corresponding to $\kmin \leq \pi/(2R)$ satisfies $\lambda_\mmin \geq 4R$.
Furthermore, we assume in the following that the receiver directions $\xhat_j \in S^{d-1}$, $j=1,\ldots,J$, are chosen such that each $d$-tuple $(\theta-\xhat_{j_1}, \ldots, \theta-\xhat_{j_d})$, $1\leq j_1<\cdots<j_d\leq J$, is linearly independent.

For a fixed receiver direction $\xhat_j$ different points $z_{m_1}\not=z_{m_2}$ might yield the same projection $(\theta-z_{m_1})\cdot\xhat_j = (\theta-z_{m_2})\cdot\xhat_j$ in \eqref{eq:emes}, and corresponding contributions in \eqref{eq:emes} might even cancel since we do not assume that all coefficients $\qtilde_{1,m}$, $m=1,\ldots,M$, have equal sign.
Hence, we introduce for each receiver direction $\xhat_j$, $j=1,\ldots,J$, sets and cardinalities\footnote{Here, $\delta_x \in \Dcal(\R)$ denotes the delta distribution with singularity in $x\in\R$.}
\begin{equation}
  \label{eq:DefMcalj}
  \Mcal_j \,:=\, \supp \Bigl( \sum_{m=1}^M (k_n\qtilde_{1,m}+\rmi \qtilde_{2,m}) \delta_{(\theta-\xhat_j)\cdot z_m} \Bigr) \tm \R \,, \qquad M_j := |\Mcal_j| \,,
\end{equation}
and
\begin{equation*}
  \Mcal_j' \,:=\, \supp \Bigl( \sum_{m=1}^M k_n\qtilde_{1,m} \delta_{(\theta-\xhat_j)\cdot z_m} \Bigr) \tm \R \,, \qquad M_j' := M_j + |\Mcal_j'| \,,
\end{equation*}
describing the (number of) exponents that (after simplification) actually appear in the non-trivial terms of the right hand side \eqref{eq:emes} for this receiver direction.
Accordingly, we rewrite the collapsed form of \eqref{eq:emes} as
\begin{equation*}
  h(\xhat_j,\theta;k_n) \,=\, \sum_{m=1}^{M_j} (k_n Q^{(j)}_{1,m} + \rmi Q^{(j)}_{2,m}) \zeta_m^n \,, \qquad n = 1,\ldots,2N \,,
\end{equation*}
where $\zeta_m := e^{\rmi \kmin f_m}$ for any $f_m \in \Mcal_j$.

Using these leading order terms, we follow the general idea of Prony-type methods and form a \emph{Hankel matrix}
\begin{equation*}
  H^{(j)} = \begin{bmatrix}
    h(\xhat_j,\theta;k_1) & h(\xhat_j,\theta;k_2) & \cdots & h(\xhat_j,\theta;k_{L+1})\\
    h(\xhat_j,\theta;k_2) & h(\xhat_j,\theta;k_3) & \cdots & h(\xhat_j,\theta;k_{L+2})\\
    \vdots   & \vdots   &        & \vdots\\
    h(\xhat_j,\theta;k_{2N-L}) & h(\xhat_j,\theta;k_{2N-L+1}) & \cdots & h(\xhat_j,\theta;k_{2N})
  \end{bmatrix}
  \in \C^{(2N-L)\times (L+1)} \,.
\end{equation*}
Using the structure of $h(\xhat_j,\theta;k_n)$ it has been shown in \cite{BadBerRic06} that the $H^{(j)}$ admits a factorization
\begin{equation*}
  H^{(j)} \,=\, V_{2N-L}^{(j)} D^{(j)} {V_{L+1}^{(j)}}^\trans \,,
\end{equation*}
where $V_l^{(j)} \in \C^{l\times M_j'}$, $l\geq2$, denotes a \emph{confluent Vandermonde matrix} that can be written as a block matrix
\begin{equation}
  \label{eq:DefV1}
  V_l^{(j)} \,=\,
  \begin{bmatrix}
    v_{1}^{(j)},\cdots, v_{M_j}^{(j)}
  \end{bmatrix}
\end{equation}
with
\begin{equation}
  \label{eq:DefV2}
  v_{m}^{(j)} \,=\, 
  \begin{bmatrix}
    1 & 0\\
    \zeta_m & 1\\
    \zeta_m^2 & 2\zeta_m\\
    \vdots &\vdots\\
    \zeta_m^{l-1} & (l-1)\zeta_m^{l-2}
  \end{bmatrix}
  \quad \text{if } Q^{(j)}_{1,m} \not= 0 \qquad \text{and} \qquad
  v_{m}^{(j)} \,=\,
  \begin{bmatrix}
    1\\
    \zeta_m\\
    \zeta_m^2\\
    \vdots \\
    \zeta_m^{l-1}
  \end{bmatrix}
  \quad \text{if } Q^{(j)}_{1,m} = 0 
\end{equation}
for $m = 1,\ldots,M_j$.
The matrix $D^{(j)} \in \C^{M_j'\times M_j'}$ is block-diagonal 
\begin{equation*}
  D^{(j)} \,=\,
  \begin{bmatrix}
    D_1^{(j)} & 0 & \cdots & 0\\
    0 & D_2^{(j)} & \ddots & \vdots\\
    \vdots & \ddots & \ddots & 0\\
    0 & \cdots & 0 & D_{M_j}^{(j)}
  \end{bmatrix}
\end{equation*}
with
\begin{align*}
  D_m^{(j)} &\,=\,
  \begin{bmatrix}
    (Q^{(j)}_{1,m} + \rmi Q^{(j)}_{2,m})\xi_m & Q^{(j)}_{1,m}\xi_m^2\\
    Q^{(j)}_{1,m}\xi_m^2 & 0
  \end{bmatrix} 
  &&\text{if } Q^{(j)}_{1,m} \not= 0\\
  \intertext{and}
  D_m^{(j)} &\,=\, [\rmi Q^{(j)}_{2,m}\xi_m] &&\text{if } Q^{(j)}_{1,m}=0 \,.
\end{align*}

\begin{remark}
  (i) If $Q^{(j)}_{1,m} = 0$ for all $m = 1,\ldots,M_j$, then the matrices $V_{2N-L}^{(j)}$ and $V_{L+1}^{(j)}$ reduce to standard Vandermonde matrices and the matrix $D^{(j)}$ is diagonal. 

  (ii) If $\qtilde_{2,m} = 0$ for all $m = 1,\ldots,M$, then we can replace \eqref{eq:emes} by
  \begin{equation*}
    \frac1{\eps^d k_n^2} \uinfty(\xhat,\theta;k_n) 
    \,\approx\, g(\xhat,\theta; k_n) 
    \,:=\, \sum_{m=1}^M \qtilde_{1,m} e^{\rmi k_n(\theta-\xhat)\cdot z_m} \,,
  \end{equation*}
  and work with $g$ instead of $h$.
  The results of the following analysis of the reconstruction scheme carry over accordingly.
  \hfill$\lozenge$
\end{remark}

Since $\zeta_1,\cdots,\zeta_{M_j}$ are mutually distinct by construction, the rank of the $V_l^{(j)} \in \C^{l\times M_j'}$ satisfies
\begin{equation}
  \label{eq:RankVl}
  \rank (V_l^{(j)}) \,=\, \min \{l, M_j'\}
\end{equation}
(see~\cite{Gau62}).
In particular, $N > L \geq M' \geq M_j'$ implies that $\rank (V_{L+1}^{(j)}) = \rank (V_{2N-L}^{(j)}) = M_j'$.
Furthermore, $D^{(j)}$ is clearly invertible, and thus
\begin{equation}
  \label{eq:RangeIdentity}
  \rank(H^{(j)}) \,=\, M_j' \qquad \text{and} \qquad \Rcal(H^{(j)}) \,=\, \Rcal(V^{(j)}_{2N-L}) \,.
\end{equation}

\begin{proposition}
  \label{pro:RangeCharacterization}
  Let $z \in B_R(0) \tm \R^d$, $\zeta := e^{\rmi \kmin (\theta-\xhat_j)\cdot z}$ and let
  \begin{equation}
    \label{eq:TestVector}
    \phi^{(j)}_z \,:=\, [1, \zeta, \zeta^2, \ldots, \zeta^{2N-L-1}]^\trans \in \C^{2N-L} \,.
  \end{equation}
  Then, $\phi^{(j)}_z \in \Rcal(H^{(j)})$ if and only if $\zeta \in \{\zeta_1,\ldots,\zeta_{M_j}\}$.
  In particular, ${\phi^{(j)}_z \in \Rcal(H^{(j)})}$ implies that ${(\theta-\xhat_j)\cdot z} \in \{{(\theta-\xhat_j)\cdot z_1}, \ldots, {(\theta-\xhat_j)\cdot z_M}\}$.
\end{proposition}

\begin{proof}
  Suppose that $\zeta \notin \{\zeta_1,\ldots,\zeta_{M_j}\}$. 
  From \eqref{eq:WaveNumbers2} we find that $2N-L \geq L+2 \geq M'+2 \geq M_j'+2$.
  Therefore, \eqref{eq:RankVl} shows that the concatenation $[V^{(j)}_{2N-L}, \phi^{(j)}_z]$ has rank $M_j'+1$, i.e., $\phi^{(j)}_z$ does not belong to the range of $V^{(j)}_{2N-L}$.
  Accordingly, the range identity \eqref{eq:RangeIdentity} implies that $\phi^{(j)}_z \notin \Rcal(H^{(j)})$.

  On the other hand, if $\zeta \in \{\zeta_1,\ldots,\zeta_{M_j}\}$, then $\phi^{(j)}_z \in \Rcal(H^{(j)})$ follows directly form \eqref{eq:RangeIdentity} together with \eqref{eq:DefV1}--\eqref{eq:DefV2}.
  Since $|\kmin (\theta-\xhat_j)\cdot y| < \pi$ for all $y \in \{z,z_1,\ldots,z_m\}$ by \eqref{eq:WaveNumbers2}, we find that
  \begin{equation*}
    \kmin (\theta-\xhat_j) \cdot z \,=\, \Im(\log\zeta) \,=\, \Im(\log(\zeta_{\tilde{m}})) \,=\, \kmin (\theta-\xhat_j) \cdot z_{m}
  \end{equation*}
  for some $1\leq\widetilde{m}\leq M_j$ and $1\leq m\leq M$.\footnote{Here $\log$ denotes the principal value of the complex logarithm.}
\end{proof}

The following theorems establish that using sufficiently many receiver directions $\xhat_j$, $j=1,\ldots,J$, we can indeed uniquely recover the positions $z_1,\ldots,z_m$ from the rescaled leading order terms $h(\xhat_j,\theta;k_n)$, $j=1,\ldots,J$ and $n = 1,\ldots,2N$.
We first assume that the coefficients $\qtilde_{1,m},\qtilde_{2,m}$, $m=1,\ldots,M$, in \eqref{eq:emes} are such that no terms in this sum cancel for any receiver direction.
This assumption will be relaxed in Theorem~\ref{thm:Identification2} below.
\begin{theorem}
  \label{thm:Identification1}
  Let $z \in B_R(0) \tm \R^d$, define $\phi^{(j)}_z$ for each receiver direction $\xhat_j$, $j=1,\ldots,J$, as in \eqref{eq:TestVector}, and let $J > (d-1)M$.
  Suppose that either
  \begin{equation}
    \label{eq:Definiteness}
    \qtilde_{1,m}>0 \quad \text{or} \quad \qtilde_{1,m}<0 \quad \text{or} \quad \qtilde_{2,m}>0 \qquad \text{simultaneously for all } m=1,\ldots,M\,.
  \end{equation}
  Then, $\phi^{(j)}_z \in \Rcal(H^{(j)})$ for all $j=1,\ldots,J$ if and only if $z \in \{z_1,\ldots,z_M\}$.
\end{theorem}

\begin{proof}
  Assume that $z \in \{z_1,\ldots,z_M\}$.
  Since no terms in \eqref{eq:emes} cancel by assumption \eqref{eq:Definiteness}, we find that $M_j = M$ in \eqref{eq:DefMcalj}, and from Proposition~\ref{pro:RangeCharacterization} we obtain that $\phi_z^{(j)} \in \Rcal(H^{(j)})$ for all $j=1,\ldots,J$.

  On the other hand, suppose that $\phi_z^{(j)} \in H^{(j)}$ for all $j=1,\ldots,J$.
  Then Proposition~\ref{pro:RangeCharacterization} shows that for each $j=1,\cdots,J$ we can find $m_j \in \{1,\ldots,M\}$ such that $(\theta-\xhat_j)\cdot z = (\theta-\xhat_j)\cdot z_{m_j}$.
  Since $J > (d-1)M$, this implies that there exist $m  \in \{1,\ldots,M\}$ and $1 \leq j_1 < \cdots < j_d \leq J$ such that
  \begin{equation*}
    (\theta-\xhat_{j_1}) \cdot z =(\theta-\xhat_{j_1}) \cdot z_m \,, \ldots , (\theta-\xhat_{j_d}) \cdot z = (\theta-\xhat_{j_d}) \cdot z_m \,.
  \end{equation*}
  Since the $d$-tuple $(\theta-\xhat_{j_1}, \ldots, \theta-\xhat_{j_d})$ is linearly independent by assumption, we can conclude that $z = z_m$.
\end{proof}

In the general case we need roughly twice as many receiver directions to uniquely recover the positions $z_1,\ldots,z_m$.
\begin{theorem}
  \label{thm:Identification2}
  Let $z \in B_R(0) \tm \R^d$, define $\phi^{(j)}_z$ for each receiver direction $\xhat_j$, $j=1,\ldots,J$, as in \eqref{eq:TestVector}, and let $J > (d-1)(2M-1)$.
  Then, $\phi^{(j)}_z \in \Rcal(H^{(j)})$ for at least $(d-1)M+1$ receiver directions $\xhat_j$ if and only if $z \in \{z_1,\ldots,z_M\}$.
\end{theorem}

\begin{proof}
  Suppose that $z \in \{z_1,\ldots,z_M\}$, and w.l.o.g.\ $z = z_1$.
  If $\phi_{z_1}^{(j)} \notin \Rcal(H^{(j)})$, then Proposition~\ref{pro:RangeCharacterization} shows that there exists $m\in\{1,\ldots,M\}$ such that $(\theta-\xhat_j)\cdot z_1 = (\theta-\xhat_j)\cdot z_m$.
  However, this can happen for $(d-1)(M-1)$ different observation direction $\xhat_j$ only, since otherwise there exist $m  \in \{2,\ldots,M\}$ and $1 \leq j_1 < \cdots < j_d \leq J$ such that
  \begin{equation*}
    (\theta-\xhat_{j_1}) \cdot z_1 = (\theta-\xhat_{j_1}) \cdot z_m \,, \ldots , (\theta-\xhat_{j_d}) \cdot z_1 = (\theta-\xhat_{j_d}) \cdot z_m \,.
  \end{equation*}
  Since the $d$-tuple $(\theta-\xhat_{j_1}, \ldots, \theta-\xhat_{j_d})$ is linearly independent by assumption, we can conclude that $z_1 = z_m$, which contradicts the assumption that $z_1, \ldots, z_m$ are mutually distinct.
  Accordingly, $\phi^{(j)}_z \in \Rcal(H^{(j)})$ for at least $(d-1)(2M-1)+1-(d-1)(M-1) = (d-1)M+1$ receiver directions $\xhat_j$, $j\in \{1,\ldots, J\}$.

  The other direction follows as in the proof of Theorem~\ref{thm:Identification1}.
\end{proof}

Introducing the matrix $F^{(j)}\in \C^{(2N-L)\times (L+1)}$ for $j=1,\ldots,J$ by
\begin{equation*}
  F^{(j)} \,=\, \begin{bmatrix}
    \frac1{k_1}\uinfty(\xhat_j,\theta;k_1) & \frac1{k_2}\uinfty(\xhat_j,\theta;k_2) & \cdots & \frac1{k_{L+1}}\uinfty(\xhat_j,\theta;k_{L+1})\\
    \frac1{k_2}\uinfty(\xhat_j,\theta;k_2) & \frac1{k_3}\uinfty(\xhat_j,\theta;k_3) & \cdots & \frac1{k_{L+2}}\uinfty(\xhat_j,\theta;k_{L+2})\\
    \vdots   & \vdots   &        & \vdots\\
    \frac1{k_{2N-L}}\uinfty(\xhat_j,\theta;k_{2N-L}) & \frac1{k_{2N-L+1}}\uinfty(\xhat_j,\theta;k_{2N-L+1}) & \cdots & \frac1{k_{2N}}\uinfty(\xhat_j,\theta;k_{2N})
  \end{bmatrix} \,,
\end{equation*}
we obtain from Theorem~\ref{thm:Asy} that
\begin{equation*}
  F^{(j)} \,=\, \eps^d H^{(j)} + O(\kmax^2\eps^{d+1}) \qquad \text{as } \kmax\eps\to0 \,,
\end{equation*}
where $\kmax = k_{2N}$ denotes the largest wave number.
The matrix  $F^{(j)}$ admits a singular value decomposition
\begin{equation*}
  F^{(j)} \,=\, \sum_{l=1}^{L+1} \sigma_l u_l v_l^* \,,
\end{equation*}
where $\sigma_1 \geq \sigma_2 \geq \cdots \geq \sigma_{L+1} \geq 0$ are the singular values of $F^{(j)}$ written in decreasing order with multiplicity, and similarly, the matrix $H^{(j)}$ can be decomposed as
\begin{equation*}
  H^{(j)} \,=\,  \sum_{l=1}^{L+1} \sigmatilde_l \utilde_l \vtilde_l^* \,,
\end{equation*}
where $\sigmatilde_1 \geq \sigmatilde_2 \geq \cdots \geq \sigmatilde_{L+1} \geq 0$ are the singular values of $H^{(j)}$ written in decreasing order with multiplicity.

Using standard perturbation results for matrices we find that (by appropriate enumeration of the singular values), as $\kmax\eps\to0$
\begin{equation*}
  \sigma_l \,=\, \eps^d \sigmatilde_l + O(\kmax^2\eps^{d+1}) \,, \qquad l=1,\ldots,L+1 
\end{equation*}
(cf.~\cite[Cor.~8.6.2]{GolLoa96}).
Since $\rank H^{(j)} = M_j'$, we expect for sufficiently small $\eps>0$ to see $M_j'$ eigenvalues of $F^{(j)}$ of the order $\eps^d$ while all the remaining eigenvalues have smaller magnitude. 
Hence, one possible strategy to estimate the numerical value of $M_j'$ may be to look for a gap in the set of singular values of $F^{(j)}$.

The orthogonal projection onto the range of $H^{j}$ is given by $\Ptilde_j \,:=\, \sum_{l=1}^{M_j'} \utilde_l \utilde_l^*$, and accordingly we denote by $P_j := \sum_{l=1}^{M_j'} u_l u_l^*$ the orthogonal projection onto the \emph{essential range} of $F^{(j)}$.
Theorem~8.6.5 in \cite{GolLoa96} implies that
\begin{equation*}
  P_j \,=\, \Ptilde_j + O(\kmax\eps) \,.
\end{equation*}

In Theorem~\ref{thm:Identification1} (assuming that no terms in \eqref{eq:emes} cancel and that $J>(d-1)M$) we have seen that a test point $z \in B_R(0)$ coincides with one of the positions $z_1,\ldots,z_M$ if and only if $\phi_z^{(j)} \in \Rcal(H^{(j)})$ for all $j=1,\ldots,J$, or equivalently if $\sum_{j=1}^J \|(I-\Ptilde_j)\phi_z^{(j)}\|_2 = 0$.
For small values of $\kmax\eps$ the projected test function $\Ptilde_j\phi_z^{(j)}$ is well approximated by $P_j\phi_z^{(j)}$, and $P_j$ can be computed by means of the singular value decomposition of the matrix $F^{(j)}$, i.e., by means of the observed far field data.
If we plot
\begin{equation}
  \label{eq:Indicator1}
  I_1(z) \,:=\, \frac1{\sum_{j=1}^J \|(I-P_j)\phi_z^{(j)}\|_2} \,, \qquad z\in B_R(0) \,,
\end{equation}
we thus expect to see peaks close to the actual positions $z_1,\ldots,z_M$.

\begin{remark}
  \label{rem:EstimateMj'1}
Evaluating $I_1$ in \eqref{eq:Indicator1} requires a priori information on the dimension of the essential range $F^{(j)}$, $j=1,\ldots,J$. 
  In particular for noisy data, estimating $M_j'$ by looking for a gap in the singular values sometimes turns out to be difficult because the gap is not clearly visible.

  In our implementation we thus follow an idea from \cite{BruHanVog03} and replace $M_j'$ in the definition of $P_j$ for all $j=1,\ldots,J$ by $\Mtilde>0$. 
Then we plot $I_1$ for increasing values of $\Mtilde$ starting with $\Mtilde=1$.
  This is reasonable, because if we replace $\Rcal(H^{(j)})$ by any $\Mtilde$-dimensional subspace $U\tm\Rcal(H^{(j)})$ then Proposition~\ref{pro:RangeCharacterization} reduces to $\phi_z^{(j)} \in U \Longrightarrow \zeta \in \{\zeta_1,\ldots,\zeta_{M_j}\}$, and accordingly the number of reconstructed scatterers should be monotonically increasing as $\Mtilde = \dim(U)$ increases until all $M$ scatterers are reconstructed for $\Mtilde\gtrsim M'_j$ for all $j=1,\ldots,J$.
  Since none of the singular vectors of $F^{(j)}$ corresponding to singular values $\sigma_l$, $l>M_j'$, is expected to be exactly of the form $\phi^{(j)}_z$, $z\notin\{z_1,\ldots,z_M\}$, the number of reconstructed objects should become stationary for moderately sized $\Mtilde\gtrsim M_j'$, $j=1,\ldots,J$.
  This procedure gives an estimate for $M$ and an upper bound $2M\geq\Ltilde\gtrsim M_j'$, $j=1,\ldots,J$.
  \hfill$\lozenge$
\end{remark}

Similarly, Theorem~\ref{thm:Identification2} (assuming that $J>(d-1)(2M-1)$) shows that a test point $z \in B_R(0)$ coincides with one of the positions $z_1,\ldots,z_M$ if and only if $\phi_z^{(j)} \in \Rcal(H^{(j)})$ for at least $(d-1)M+1$ receiver directions $\xhat_j$, $j\in\{1,\ldots,J\}$, or equivalently, if $\sum_{j\in\Jcal} \|(I-P_j)\phi_z^{(j)}\|_2 = 0$, where $\Jcal\tm\{1,\ldots,J\}$ is the index set corresponding to the $(d-1)M+1$ smallest elements in $\{ \|(I-P_j)\phi_z^{(j)}\|_2 \; : \; j=1,\ldots,J\}$.
Accordingly, we expect the imaging functional
\begin{equation}
  \label{eq:Indicator2}
  I_2(z) \,:=\, \frac1{\sum_{j\in\Jcal} \|(I-P_j)\phi_z^{(j)}\|_2} \,, \qquad z\in B_R(0) \,,
\end{equation}
to peak close to the actual positions $z_1,\ldots,z_M$.

\begin{remark}
  \label{rem:EstimateMj'2}
  Evaluating $I_2$ in \eqref{eq:Indicator2} requires a priori information on the number $M$ of unknown scatterers and on the dimension $M_j'$ of the essential range $F^{(j)}$, $j=1,\ldots,J$. 
  We proceed as in Remark~\ref{rem:EstimateMj'1} to obtain an estimate for $M$ and an upper bound $2M\geq\Ltilde\gtrsim M_j'$, $j=1,\ldots,J$, from the given far field observations.
  \hfill$\lozenge$
\end{remark}

\section{Numerical examples}
\label{sec:Numerics}
In this section we provide a three-dimensional numerical example to illustrate the performance of the multifrequency MUSIC reconstruction scheme, and we apply the algorithm to backscattering data.

\begin{example}
  \label{exa:1}
  We consider three ellipsoidal scattering objects centered at positions $z_1=(2,2,2)$, $z_2=(-1,-3,-1)$ and $z_3=(-3,1,2)$ with semiaxes $(0.2,0.2,0.1)$, $(0.1,0.1,0.2)$ and $(0.1,0.1,0.1)$, respectively, as shown in Figure~\ref{fig:NumEx1} (left).
  \begin{figure}
    \centering
    \includegraphics[height=5.0cm]{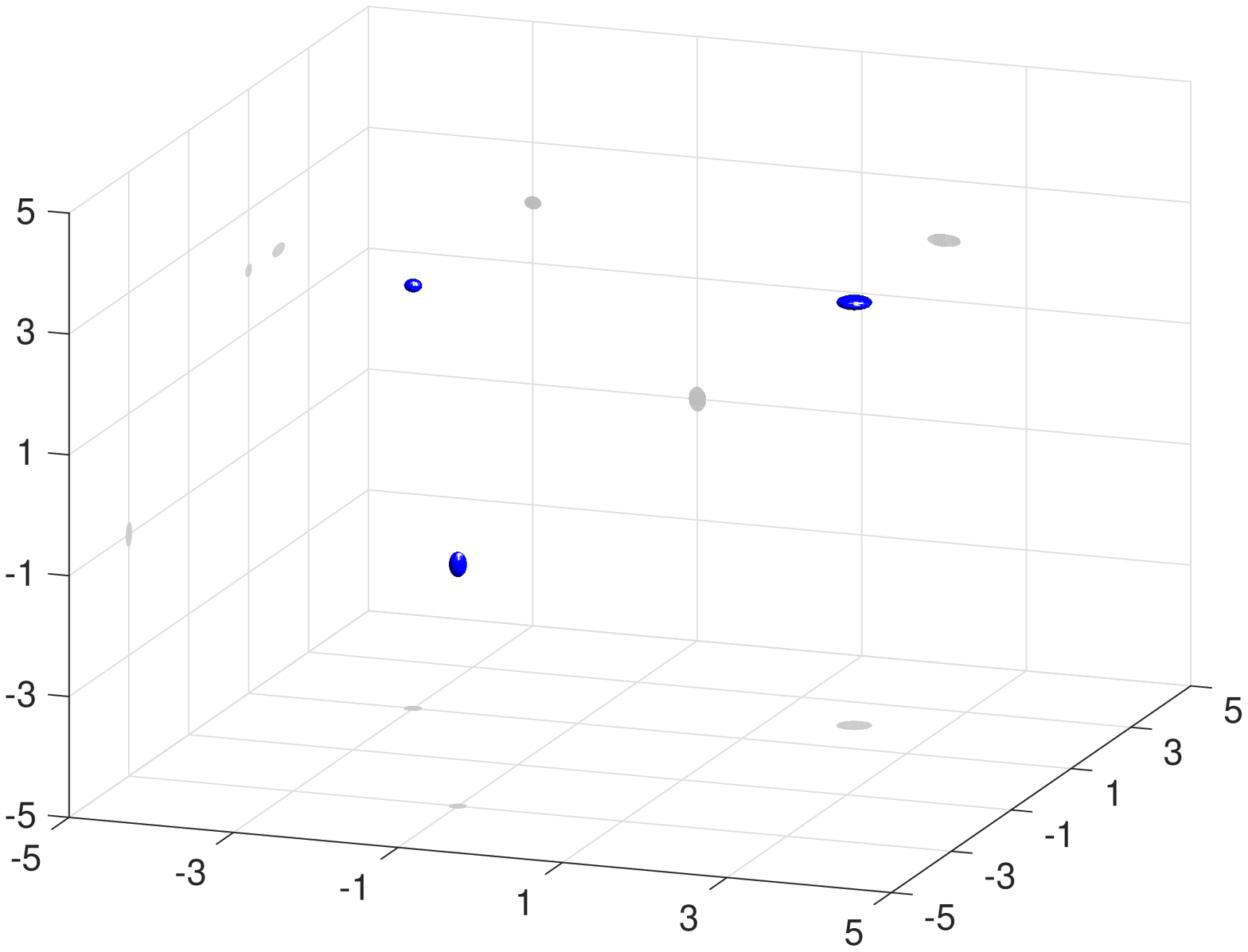}
    \qquad
    \includegraphics[height=5.0cm]{plots/3dmedium_recdir.eps}
    \caption{\small Left: Geometry of the scatterers.
      Right: Map projection of the observation directions.} 
    \label{fig:NumEx1}
  \end{figure}
  Accordingly the whole ensemble of scatterers is contained in the ball $B_R(0)$ with $R=5$.
  The index of refraction of the scattering objects shall be given by $n_1=0.5+0.5\rmi/k$, $n_2=2$ and $n_3=3+\rmi/k$, where as before $k>0$ denotes the wave number.

  Following our discussion in \eqref{eq:WaveNumbers1}--\eqref{eq:WaveNumbers2} we choose $\kmin = \pi/10$, $N=16$, and accordingly consider $32$ wavenumbers $k_n=n\pi/10$, $n=1,\ldots,32$ (i.e., $\kmin \approx 0.314$ and $\kmax \approx 10.053$ and accordingly the wavelength varies between $\lambdamin=0.625$ and $\lambdamax=20$).
  For the upper bound on $M'$ from \eqref{eq:DefM'} in the reconstruction algorithm we use $L=15$.

  We consider $32$ incident fields $u^i(\cdot,\theta;k_n)$, $n=1,\ldots,32$, with fixed incident direction $\theta = (1,0,0)$ and simulate the corresponding far field patterns $\uinfty(\xhat_j,\theta;k_n)$ at $12$ randomly distributed receiver directions $\xhat_j \in S^2$ such that each triple $(\theta-\xhat_{j_1}, \theta-\xhat_{j_2}, \theta-\xhat_{j_3})$, $1\leq j_1<j_2<j_3\leq 12$, is linearly independent using the C++ boundary element library BEM++ (cf.~\cite{SmiBetArrPhiSch15}).
  A map projection of the receiver directions is shown in Figure~\ref{fig:NumEx1} (right).
  In addition to the numerical error we perturb the simulated far field data by a uniformly distributed relative additive random noise of $10$\%.
  In total the data set consists of $384$ far field observations.

  First we use the values of the indicator function $I_1(z)$ from \eqref{eq:Indicator1} for $z \in [-5,5]^3$ to visualize the location of the scatterers.
  \begin{figure}
    \centering
    \includegraphics[height=5.0cm]{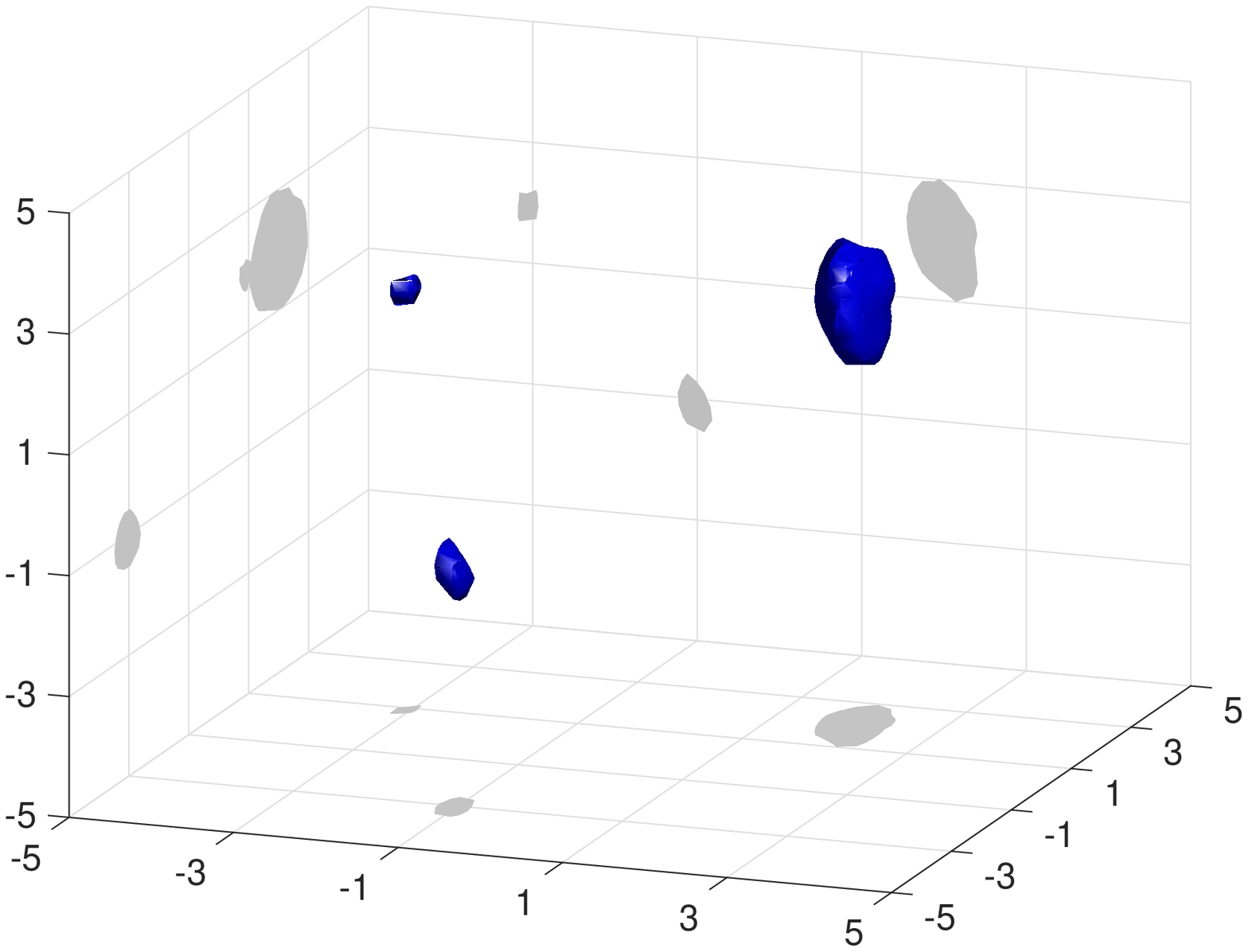}
    \qquad
    \includegraphics[height=5.0cm]{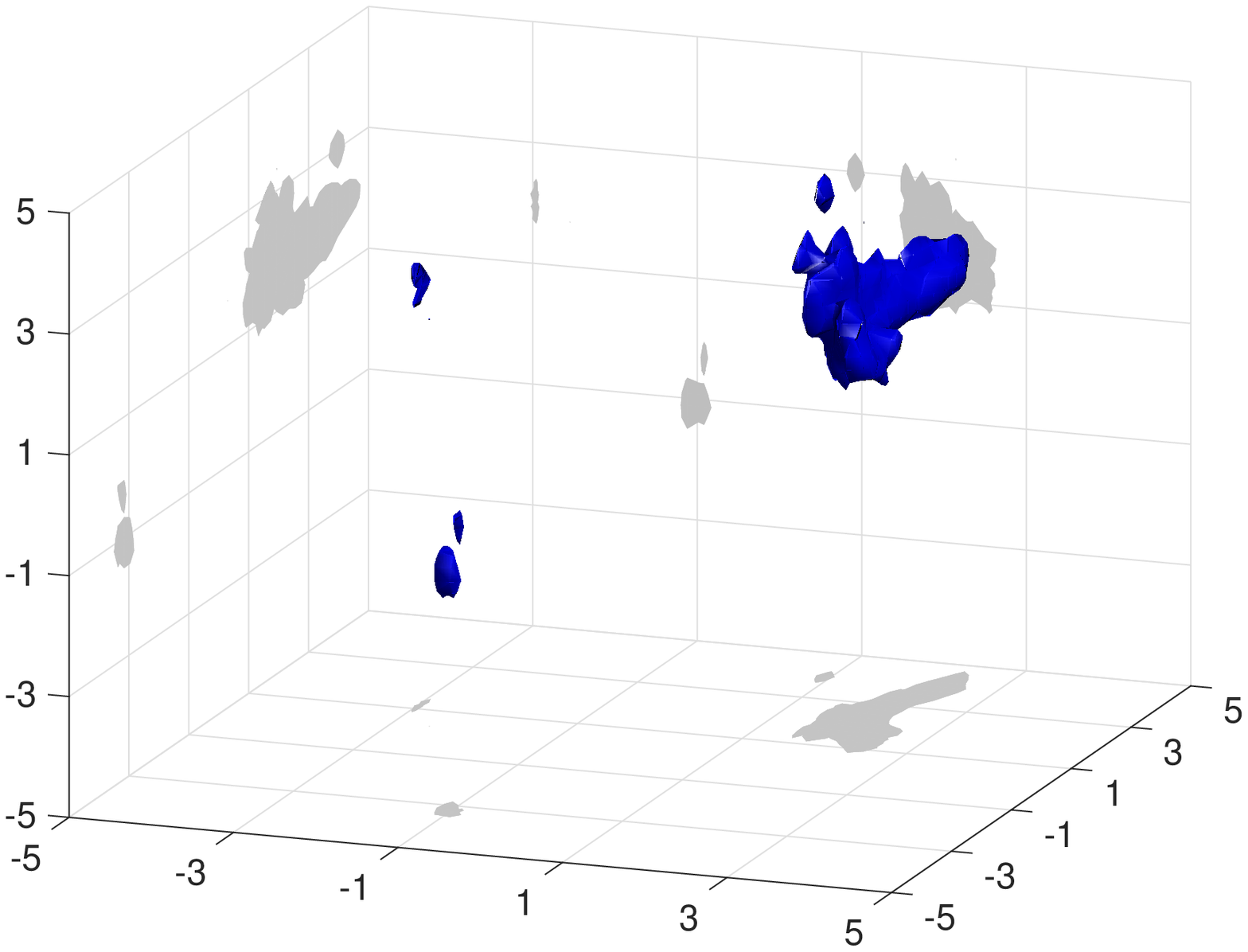}
    \caption{\small Reconstructions for Example~\ref{exa:1} (incident direction $\theta = (1,0,0)$): Isosurfaces of $I_1$ (left) and $I_2$ (right) at $55\%$ and $58\%$ of their respective maximum value.}
    \label{fig:NumEx2}
  \end{figure}
  Figure~\ref{fig:NumEx2} (left) shows an isosurface plot of $I_1$ at $55\%$ of its maximum value, where we used $\Ltilde = 6$ instead of $M'_j$ in the definition of the projections $P_j$, $j=1,\ldots,12$.
  This value has been determined by the iterative procedure outlined in Remark~\ref{rem:EstimateMj'1}.
  Note that the plots in Figure~\ref{fig:NumEx2} not only visualize the reconstructed positions of the scatterers but also their projections onto the boundaries of the box to enhance the three dimensional perspective.
  Moreover, we emphasize that the isosurfaces should not be mistaken as reconstructions of the shape of the scatterers; however they indeed contain the true scatterers and determine their positions rather accurately.

  An isosurface plot of $I_2$ at $58\%$ is shown in Figure~\ref{fig:NumEx2} (right), where we used $M=3$ and $\Ltilde = 6$.
  Again these values have been determined by the iterative procedure described in Remark~\ref{rem:EstimateMj'1}.
  The isosurface plots for $I_2$ are considerably less smooth than the corresponding plots for $I_1$ in all our numerical tests.
  \hfill$\lozenge$
\end{example}

In our second example below we apply the multifrequency MUSIC reconstruction algorithm to multifrequency backscattering data, i.e., instead of the data set from \eqref{eq:MeasurementData} we assume that backscattered far field data
\begin{equation*}
  \uinfty(-\xhat_j,\xhat_j;k_n) \,, \qquad j=1,\ldots,J \,,\; n=1,\ldots,2N \,,
\end{equation*}
for $J>0$ mutually distinct \emph{incident/receiver directions} $\xhat_j \in S^{d-1}$ and $2N>0$ different wave numbers $0 < k_1 < \cdots < k_{2N}$ are observed.
The main attraction of this backscattering measurement setup is that only a single sensor which acts both as source and receiver is required in practice.

Similar to \eqref{eq:emes}, Theorem~\ref{thm:Asy} implies that, for $j=1,\ldots,J$ 
\begin{equation*}
  \frac1{\eps^d k_n} \uinfty(-\xhat_j,\xhat_j;k_n) 
  \,\approx\, h(-\xhat_j,\xhat_j; k_n) 
  \,:=\, \sum_{m=1}^M (k_n \qtilde_{1,m} + \rmi \qtilde_{2,m}) e^{2 \rmi k_n \xhat_j\cdot z_m}
\end{equation*}
if $k_n\eps$ is sufficiently small.
Accordingly, assuming that the incident/receiver directions $\xhat_1,\ldots,\xhat_J \in S^2$ are such that each $d$-tuple $(\xhat_{j_1}, \ldots, \xhat_{j_d})$, $1\leq j_1<\cdots<j_d\leq J$, is linearly independent, and that the wave numbers $k_1,\ldots,k_{2N}$ satisfy \eqref{eq:WaveNumbers1}--\eqref{eq:WaveNumbers2}, then the multifrequency MUSIC reconstruction algorithm from Section~\ref{sec:MUSIC} can be applied to backscattering data without changes, and its theoretical justification including the Theorems~\ref{thm:Identification1}--\ref{thm:Identification2} remains valid in this case.

\begin{example}
  \label{exa:2}
  We consider the same scattering objects with the same material properties as in the previous example, and use the $12$ receiver directions $\xhat_1,\ldots,\xhat_{12}$ from Example~\ref{exa:1} as incident/receiver directions for the backscattering data in this example (we note that each triple $(\xhat_{j_1},\xhat_{j_2},\xhat_{j_3})$, $1\leq j_1<j_2<j_3\leq 12$, is linearly independent).
  Moreover we choose $\kmin = \pi/10$, $N=16$, and accordingly consider $32$ wavenumbers $k_n=n\pi/10$, $n=1,\ldots,32$, and we use $L=15$ for the upper bound on $M'$ in the reconstruction algorithm (same as in Example~\ref{exa:1}). 
  We simulate the backscattered far field patterns $\uinfty(-\xhat_j,\xhat_j;k_n)$, $j=1,\ldots,12$, $n=1,\ldots,32$, using BEM++ and add uniformly distributed relative additive random noise of $10$\%.
  In total again $384$ far field observations are used.

  \begin{figure}
    \centering
    \includegraphics[height=5.0cm]{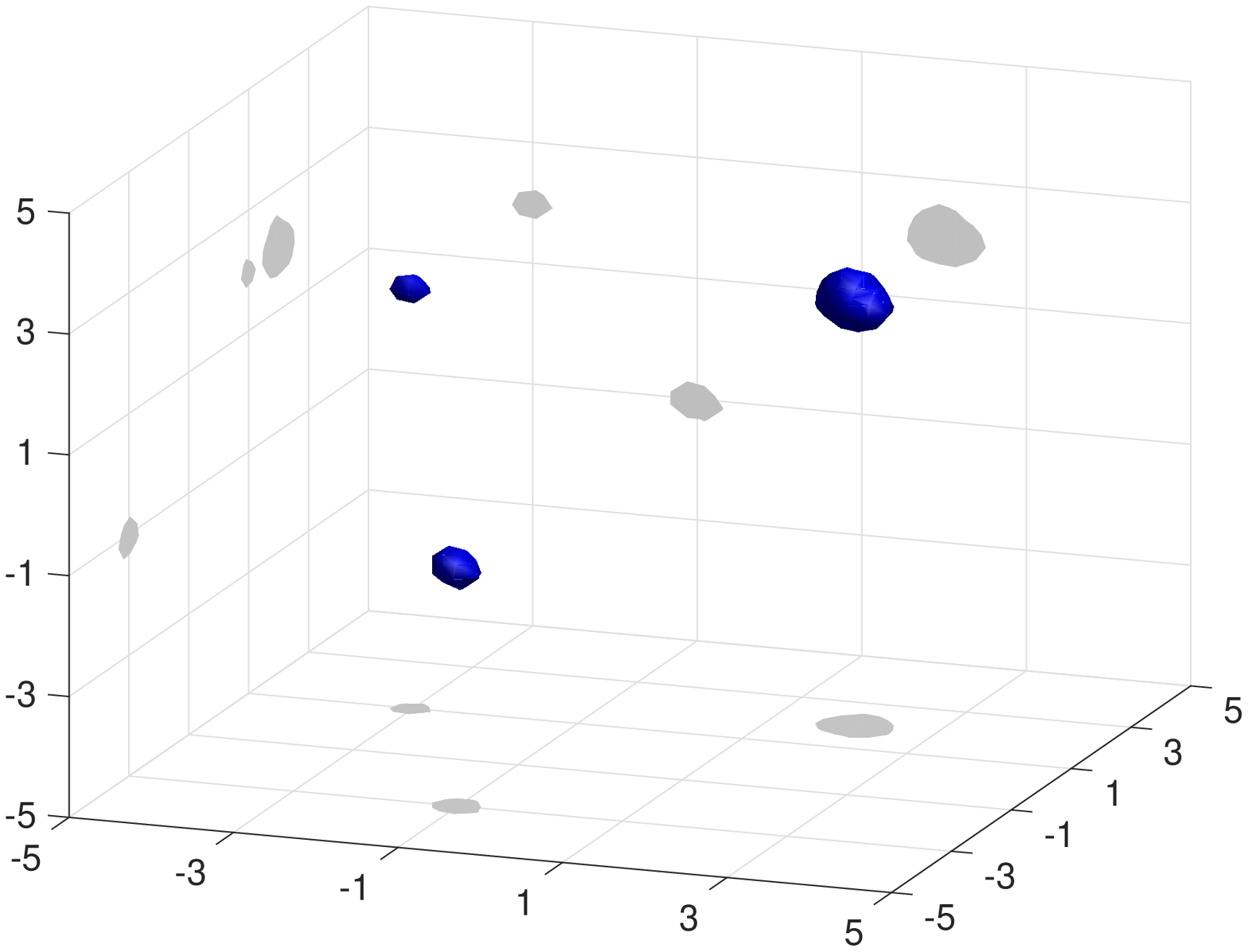}
    \qquad
    \includegraphics[height=5.0cm]{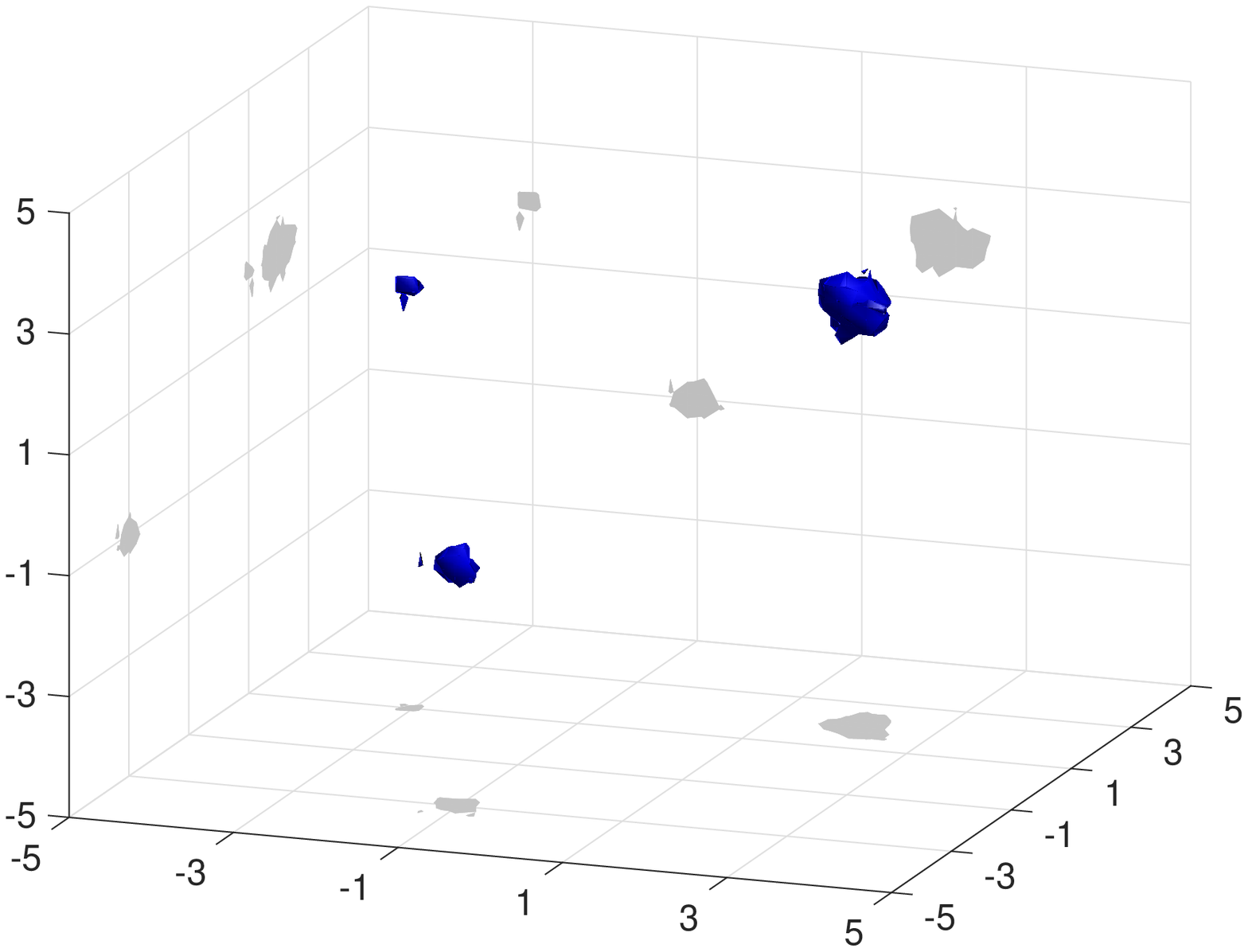}
    \caption{\small Reconstructions for Example~\ref{exa:2} (backscattering): Isosurfaces of $I_1$ (left) and $I_2$ (right) at $45\%$ and $52\%$ of their respective maximum value.}
    \label{fig:NumEx3}
  \end{figure}
  An isosurface plot of the indicator function $I_1$ from \eqref{eq:Indicator1} at $45\%$ of its maximum value is shown in Figure~\ref{fig:NumEx3} (left).
  Here we used $\Ltilde = 6$ instead of $M'_j$, $j=1,\ldots,12$.
  This value has been determined by the iterative procedure outlined in Remark~\ref{rem:EstimateMj'1}.
  Figure~\ref{fig:NumEx3} (right) shows an isosurface plot of $I_2$ at $52\%$, where we used $M=3$ and $\Ltilde = 6$.
  Again these values have been determined by the iterative procedure described in Remark~\ref{rem:EstimateMj'1}.
  A comparison of the plots from Figure~\ref{fig:NumEx2} (single incident wave) with Figure~\ref{fig:NumEx3} (backscattering data) does not show a significant difference; the reconstructions obtained from backscattering data seem to be a little more focused.
  \hfill$\lozenge$
\end{example}

\section{Concluding remarks}
\label{sec:Conclusion}
We have considered a qualitative reconstruction method for an inverse acoustic or electromagnetic medium scattering problem that efficiently processes multifrequency information.
We have given explicit lower bounds on the number of frequencies and on the incident/receiver directions that are required to reconstruct a given configuration of finitely many point-like scatterers.
The results and the reconstruction scheme apply to backscattering data as well.

We are currently working on extensions of this scheme for Maxwell's equations and consider possible generalizations of these techniques to recover extended (non-small) scattering objects from multifrequency (back-)scattering data.




\end{document}